\documentclass[11pt,reqno]{amsart}
\usepackage[margin=1in]{geometry}
\usepackage{amssymb,amsmath,amsthm}
\usepackage{graphicx}
\usepackage{subcaption}
\usepackage{fullpage}
\usepackage{scrextend}
\usepackage{siunitx}
\usepackage{mathtools}
\usepackage[dvipsnames]{xcolor}
\usepackage{microtype}
\usepackage{enumerate}
\usepackage{tikz,calc}
\usepackage{tikz-cd}
\usetikzlibrary{automata,positioning}
\usepackage{comment}
\usetikzlibrary{calc}
\usepackage{epsfig,graphicx}
\usepackage{listings}
\usepackage{enumitem}

\usepackage{float}
\usepackage{bbm}
\usepackage[colorlinks=true, pdfstartview=FitH, linkcolor=NavyBlue, citecolor=RawSienna, urlcolor=RoyalBlue]{hyperref}

\patchcmd{\subsubsection}{\itshape}{\bfseries}{}{}

\newcommand{\custMR}[1]{\href{http://www.ams.org/mathscinet-getitem?mr=#1}{MR#1}}
\newcommand{\arxiv}[1]{\href{http://arxiv.org/abs/#1}{arXiv:#1}}
\newcommand{\defeq}{\vcentcolon=}

\newcommand{\brax}[1]{\left( #1 \right)}
\newcommand{\abs}[1]{\left|#1\right|}

\renewcommand{\le}{\leqslant}
\renewcommand{\geq}{\geqslant}
\renewcommand{\ge}{\geqslant}

\newcommand{\N}{\mathbb{N}}
\newcommand{\R}{\mathbb{R}} 
\newcommand{\Z}{\mathbb{Z}}

\newcommand{\Mod}[1]{\ (\text{mod}\ #1)}
\renewcommand{\Mod}[1]{{\ifmmode\text{\rm\ (mod~$#1$)}\else\discretionary{}{}{\hbox{ }}\rm(mod~$#1$)\fi}}

\newcommand{\allnotes}[1]{}
\renewcommand{\allnotes}[1]{\textit{#1}}

\newtheorem{theorem}{Theorem}[section]
\newtheorem{prop}[theorem]{Proposition}
\newtheorem{lemma}[theorem]{Lemma}
\newtheorem{cor}[theorem]{Corollary}

\newtheorem{conjecture}[theorem]{Conjecture}

\theoremstyle{definition} % changes the formatting of the environments defined below

\newtheorem*{note}{Note}
\numberwithin{theorem}{section}

\allowdisplaybreaks

\begin{document}

% Spacing between paragraphs
\parskip \smallskipamount

\title{On subsets of lattice cubes avoiding affine and spherical degeneracies}

\author{Anubhab Ghosal}
\author{Ritesh Goenka}
\author{Peter Keevash}
\address{Mathematical Institute \\ University of Oxford \\ Oxford, UK OX2 6GG}
\email{\{ghosal,goenka,keevash\}@maths.ox.ac.uk}

\subjclass[2020]{05D40, 52C10, 52C35}
\keywords{no-four-on-a-circle, no-three-in-line, evasive sets.}

\begin{abstract}
    For integers $1 < k < d-1$ and $r \ge k+2$, we establish new lower bounds on the maximum number of points in $[n]^d$ such that no $r$ lie in a $k$-dimensional affine (or linear) subspace. These bounds improve on earlier results of Sudakov-Tomon and Lefmann. Further, we provide a randomised construction for the no-four-on-a-circle problem posed by Erd\H{o}s and Purdy, improving Thiele's bound. We also consider the random construction in higher dimensions, and improve the bound of Suk and White for $d \geq 4$. In each case, we apply the deletion method, using results from number theory and incidence geometry to solve the associated counting problems.
\end{abstract}

\maketitle

\section{Introduction}
\label{sec:intro}

At the beginning of the $20$th century, Dudeney~\cite{Dud} posed the following problem: How many points can be placed in an $n \times n$ integer grid so that no three points lie in a straight line? This is now well known as the \emph{no-three-in-line} problem and has been extensively studied (see \cite[Section~10.1]{BMP} and \cite[F4]{Guy}). Several similar problems have been studied since then, namely, problems that seek to maximize the size of subsets of lattice cubes while avoiding specific configurations. Such problems are fundamental not only in discrete geometry but also relevant across other fields. For instance, higher dimensional finite field analogues of no-three-in-line sets provide examples of \emph{evasive sets}~\cite{PR} (see also \cite{ST}), which provide explicit constructions of list-decodable error correcting codes~\cite{Gur} and bipartite Ramsey graphs~\cite{PR}, and are also related to the \emph{cap set} problem~\cite{CLP,EG} from extremal combinatorics. The no-three-in-line problem can also be viewed as a discrete analogue of Heilbronn's triangle problem~\cite[F4]{Guy} (see \cite{CPZ} for recent progress on the latter) from discrepancy theory, which asks for a placement of points in the unit square that maximizes the area of the smallest triangle.

Brass and Knauer~\cite{BK} asked the following generalized problem: given $n, d, k, r \in \N$ with $k < d$ and $r \ge k + 2$, what is the maximum number $f_{\text{aff}}(n, d, k, r)$ of points in $[n]^d$ such that no $r$ of these points lie on a $k$-dimensional affine subspace? The case $r = k+2$ \cite[Section 10.1, Problem 5]{BMP} is of special interest since it corresponds to sets that avoid up to $k$-dimensional affine degeneracies. Further specializing to $k = d-1$ \cite[Section~10.1, Problem 3]{BMP} is even more interesting since it corresponds to sets in general linear position. We will prove the following lower bound on $f_{\text{aff}}(n, d, k, r)$.

\begin{theorem}
\label{thm:affine}
    Let $d, k, r \in \N$ with $k < d$ and $r \ge k + 2$. Then $f_{\mathrm{aff}}(n, d, k, r) = \Omega(f_{d, k, r}(n))$, where
    \begin{equation*}
        f_{d, k, r}(n) = \begin{cases}
            n^{d(1-k/(r-1))}, &\text{ if } r \le d,\\
            n^{d-k}/ (\log n)^{1/d}, &\text{ if } r = d+1,\\
            n^{d-k}, &\text{ otherwise}.
        \end{cases}
    \end{equation*}
    In particular, if $r = k+2$, then $f_{\mathrm{aff}}(n, d, k, k+2) = \Omega(f_{d, k}(n))$, where
    \begin{equation*}
        f_{d, k}(n) = \begin{cases}
            n/(\log n)^{1/d}, &\text{ if } k = d-1,\\
            n^{d/(k+1)}, &\text{ otherwise}.
        \end{cases}
    \end{equation*}
\end{theorem}

An obvious upper bound for $f_{\text{aff}}(n, d, k, r)$ is $(r-1) n^{d-k}$ as can be seen by partitioning $[n]^d$ into $n^{d-k}$ translates of $[n]^k\times\{0\}^{d-k}$. Brass and Knauer~\cite[Section~2]{BK} noted that lower bounds matching up to a constant are known in some special cases:
\begin{itemize}
    \item for $k = 1$ and $r > d$, as can be seen by considering points on the modular moment surface $x_1 + x_2^2 + \dots + x_d^d \equiv q \mod p$, and
    
    \item for $k = d-1$, as can be seen by taking points on the modular moment curve $\{t \mapsto (t, t^2, \dots, t^d) \mod p\}$,
\end{itemize}
where $p$ is a prime between $\lceil n/2 \rceil$ and $n$. A matching lower bound was also obtained by P\'or and Wood~\cite{PW} for the case $d = 3,\, k = 1,\, r = 3$ by taking a suitable positive density subset of points on the modular paraboloid $\{(x_1, x_2) \mapsto (x_1, x_2, x_1^2 + x_2^2) \mod p\}$. For $1 < k < d-1$, Brass and Knauer~\cite[Lemma~8]{BK} showed that for each $\varepsilon > 0$, there exists an integer $r \ge k + 2$ such that $f_{\text{aff}}(n, d, k, r) = \Omega(n^{d-k-\varepsilon})$. More recently, Sudakov and Tomon~\cite[Theorem~1.4]{ST} improved this result by showing $f_{\text{aff}}(n, d, k, r) = \Theta(n^{d-k})$ for $r > d^k$. Theorem~\ref{thm:affine} implies $f_{\text{aff}}(n, d, k, r) = \Theta(n^{d-k})$ for $r > d+1$, thereby improving upon Sudakov and Tomon's result.

As noted above, algebraic constructions are known to yield improved lower bounds in some special cases. Similarly, taking lattice points on a strictly convex surface yields a lower bound of $\Omega(n^{d-2+2/(d+1)})$ for no-three-in-line sets in $[n]^d$ \cite{BL} (see also \cite{Elk}). This is also the best possible bound that can be obtained using such a construction since the maximum number of points in $[n]^d$ lying on a strictly convex surface is $O(n^{d-2+2/(d+1)})$ by a result of Andrews~\cite{And}. This bound is better than the one in Theorem~\ref{thm:affine} for $k = 1,\, 3 \le r \le d/2+1$. For the remaining subcases of $k = 1$, namely, $d/2 + 1 < r \le d$, Lefmann~\cite{Lef1} (see also \cite{Lef2}) showed that it is possible to improve the bound in Theorem~\ref{thm:affine} by a polylog factor using the results of Ajtai, Koml\'os, Pintz, Spencer and Szemer\'edi~\cite{AKPSS} on independent sets in hypergraphs. For all the remaining cases, namely, $1 < k < d-1$, Theorem~\ref{thm:affine} improves upon the previously best known bound of $\Omega(n^{d-k-k(d+1)/(r-1)})$ by Lefmann~\cite{Lef1}. Finally, we remark that a non-trivial upper bound of $O(n^{\frac{d}{\lfloor (k+2)/2 \rfloor}})$ on $f_{\text{aff}}(n, d, k, k+2)$ was obtained by Lefmann~\cite{Lef1}, and was subsequently improved to $O(n^{\frac{d}{2 \lfloor (k+2)/4 \rfloor}(1 - \frac{1}{2 \lfloor (k+2)/4 \rfloor d + 1})})$ by Suk and Zeng~\cite{SZ}.

Brass and Knauer~\cite{BK} also asked the following related problem: given $n, d, k, r \in \N$ with $k < d$ and $r \ge d+1$, what is the maximum number $f_{\mathrm{lin}}(n, d, k, r)$ of points from $[n]^d$ such that any $k$-dimensional linear subspace contains at most $r-1$ of these points? For $k = 1$, this number is (up to a constant) equal to the number of primitive lattice points in $[n]^d$, which is well known to be $\Theta(n^d)$. In general, Balko et al.~\cite{BCV} showed that $f_{\mathrm{lin}}(n, d, k, r) = O(n^{d(d-k)/(d-1)})$. For $k = d-1$, B\'ar\'any et al.~\cite{BHPT} proved that $f_{\mathrm{lin}}(n, d, d-1, r) = \Theta(n^{d/(d-1)})$. Brass and Knauer~\cite[Conjecture~9]{BK} conjectured that $f_{\mathrm{lin}}(n, d, k, k+1) = \Theta(n^{d(d-k)/(d-1)})$. However, Lefmann~\cite{Lef2} refuted their conjecture for most values of $k$ and $d$ by showing that $f_{\mathrm{lin}}(n, d, k, k+1) = O(n^{d/\lceil k/2 \rceil})$. More recently, Sudakov and Tomon~\cite[Theorem~1.5]{ST} showed that $f_{\text{lin}}(n, d, k, r) = \Theta(n^{d(d-k)/(d-1)})$ for $r > d^k$. We will prove the following theorem, which recovers the result of B\'ar\'any et al.~\cite{BHPT} and extends the result of Sudakov and Tomon~\cite{ST} to a larger regime for $r$, but only up to polylog factors.

\begin{theorem}
\label{thm:linear}
    Let $d, k, r \in \N$ with $k < d$ and $r \ge k + 1$. Then $f_{\mathrm{lin}}(n, d, k, r) = \Omega(g_{d, k, r}(n))$, where
    \begin{equation*}
        g_{d, k, r}(n) = \begin{cases}
            n^{d(r-k)/(r-1)}, &\text{ if } r < d,\\
            n^{d(d-k)/(d-1)} / (\log n)^{1/(d-1)}, &\text{ otherwise}.
        \end{cases}
    \end{equation*}
    In particular, if $r = k+1$, then $f_{\mathrm{lin}}(n, d, k, k+1) = \Omega(h_{d, k}(n))$, where
    \begin{equation*}
        g_{d, k}(n) = \begin{cases}
            n^{d(d-k)/(d-1)} / (\log n)^{1/(d-1)}, &\text{ if } k = d-1,\\
            n^{d/k}, &\text{ otherwise}.
        \end{cases}
    \end{equation*}
\end{theorem}

Erd\H{o}s and Purdy (see \cite[F3]{Guy}) asked the \emph{no-four-on-a-circle} problem: What is the maximum number $f_{\text{circ}}(n)$ of points that can be placed in an $n \times n$ integer grid so that no four points lie on a circle, where a line is also considered to be a degenerate circle? Guy~\cite{Guy} mentions that Erd\H{o}s and Purdy proved $f_{\text{circ}}(n) = \Omega(n^{2/3-o(1)})$. Thiele~\cite{Thi2} (see also \cite{Thi1}) subsequently improved this bound to $f_{\text{circ}}(n) > n/4$ by taking a suitable subset of a modular parabola. This matches (up to a constant) the obvious upper bound of $3n$ obtained by partitioning $[n]^2$ into $n$ vertical lines. By forbidding isosceles trapezia with a horizontal axis of symmetry, Thiele~\cite{Thi1} proved $f_{\text{circ}}(n) < 5n/2$. We improve on the lower bound of Thiele using a randomised construction. To do so, we determine the asymptotic count of concyclic quadruples in $[n]^2$, a quantity interesting in its own right.

\begin{theorem}
\label{thm:circle-count}
    For any $\varepsilon > 0$, the number of cyclic quadrilaterals with vertices in $[n]^2$ is $\gamma n^5 + O(n^{4+\frac{18}{29}+\varepsilon})$, where $\gamma$ is as defined in \eqref{eqn:c}.
\end{theorem}

\begin{cor}
\label{thm:circle}
    Suppose $n$ is large enough. Then there are at least $7n/12$ points in $[n]^2$ such that no four are collinear or concyclic, that is, $f_{\text{circ}}(n) \ge 7n/12$.
\end{cor}

The proof of Corollary~\ref{thm:circle} is computer-assisted. We use a computer program (see Lemma~\ref{lem:c}) to obtain a rigorous approximation of the constant $\gamma$ appearing in Theorem~\ref{thm:circle-count}, which implies a bound with a constant slightly larger than $7/12$ in Corollary~\ref{thm:circle}.

Thiele~\cite{Thi1} asked the following high dimensional version of the no-four-on-a-circle problem in his thesis: What is the maximum number $f_{\text{sph}}(n,d)$ of points in $[n]^d$ so that no $d+2$ points lie on a $(d-1)$-dimensional sphere, where a hyperplane is also considered to be a degenerate sphere? Thiele proved an $\Omega(n^{1/(d-1)})$ lower bound by considering a suitable subset of the modular moment curve. The same problem was also subsequently asked by Brass, Moser, and Pach in their monograph~\cite[Section~10.1, Problem~4]{BMP}. Suk and White~\cite{SW} recently improved Thiele's bound to $\Omega(n^{3/(d+1)-o(1)})$ for $d \ge 3$ using a randomised construction. We prove the following bounds on the count $S(n, d)$ of $(d+2)$-tuples of points in $[n]^d$ that lie in a $(d-1)$-dimensional sphere.

\begin{theorem}
\label{thm:sphere-count}
    There is a constant $c$ such that $n^{d^2+d-2}\lesssim S(n, d)\lesssim n^{d^2+2d - 4 + c/\log \log n} + n^{d^2+d} \log n$ for every integer $d \ge 3$.
\end{theorem}

As a corollary, we obtain the following improvement to Suk and White's result.

\begin{cor}
\label{thm:sphere}
    There is a constant $c$ such that $f_{\mathrm{sph}}(n,d) = \Omega\brax{n^{\frac{\min\{d,4\}}{d+1} - \frac{c}{\log \log n}}}$ for all $d \ge 3$.
\end{cor}

Suk and White~\cite[Conjecture~1.2]{SW} conjectured that $f_{\text{sph}}(n,d) = \Omega(n^{d/(d+1)})$ for $d \ge 3$. Their results imply this bound for $d = 3$ up to a $o(1)$ term in the exponent. Similarly, Corollary~\ref{thm:sphere} implies this bound for $d = 4$ up to a $o(1)$ term in the exponent.

Note the gap between lower and upper bounds in Theorem~\ref{thm:sphere-count}. We conjecture that the true order of growth of $S(n,d)$ is close to the lower bound -- see Conjecture~\ref{conj:sphere-counting} and the related discussion in Section~\ref{sec:improvedspherecount}. Together with an estimate on the number of $(d+2)$-tuples of points lying on hyperplanes in $[n]^d$ (Proposition~\ref{prop:affine-estimate}), this would imply $f_{\text{sph}}(n) = \Omega(n)$.

\subsection{Random constructions for extremal lattice problems}
\label{sec:random}

Guy and Kelly~\cite{GK} showed that the number of collinear triples of points in $[n]^2$ is equal to $\Theta(n^4 \log n)$, which, by the deletion method, yields a set of size $\Omega(n/\sqrt{\log n})$ with no three collinear points. This bound can be improved further to $\Omega(n \sqrt{\log \log n/\log n})$ by using a result of Cooper and Mubayi~\cite{CM} about the chromatic number of $k$-uniform hypergraphs (see also \cite{Cle}). While algebraic constructions of size $\Omega(n)$ are known for the no-three-in-line problem, it is not known whether a randomised construction of this size exists. Green~\cite[Problem 72]{Gre} raised the question whether any no-$3$-in-a-line-subset of large enough size reduces to an algebraic curve modulo some prime. A construction exhibiting some weak pseudorandomness properties could give a negative answer to Green's question. See also Eppstein's blog post \cite{Epp} for a discussion of randomised constructions of no-three-in-line sets. For the related no-four-on-a-circle problem, Corollary \ref{thm:circle} shows that random constructions of size $\Omega(n)$ exist. This would also be true in higher dimensions if Conjecture \ref{conj:sphere-counting} holds.

\subsection{Methods}

All the problems we consider involve determining the largest size of an independent set $\alpha(\mathcal{H})$ of some $r$-uniform hypergraph $\mathcal{H}$, with vertices $V(\mathcal{H})=[n]^d$ and $r$ vertices forming an edge if and only if they lie on a circle, a sphere, an affine subspace or a linear subspace. The deletion method yields the following lower bound on $\alpha(\mathcal{H})$.

\begin{lemma}[Spencer ~\cite{Spe}]
\label{lem:deletion}
    For any $r$-uniform hypergraph $\mathcal{H}$, we have $\alpha(\mathcal{H})\geq \frac{r-1}{r^{\frac{r}{r-1}}}\frac{|V(\mathcal{H})|^{\frac{r}{r-1}}}{|E(\mathcal{H})|^\frac{1}{r-1}}$.
\end{lemma}

To prove each of our extremal results, we will solve the corresponding counting problem, that is, count the number of edges in the associated hypergraph, and then use the above lemma.

\subsubsection{Counting cyclic quadrilaterals}

The main ingredient in the proof of Theorem~\ref{thm:circle-count} is Huxley and Konyagin's~\cite{HK} result that circles passing through four lattice points are rare among the circles that pass through three lattice points. Their proof proceeds via the unique factorisation of Gaussian integers. In particular, similar approaches are unlikely to succeed in higher dimensions.

\subsubsection{Counting conspheric tuples}

The main technical ingredient behind the proof of Theorem \ref{thm:sphere-count} is a result of Lund \cite{bendlund} that bounds the number of incidences between points and $k$-dimensional affine spaces in $\R^d$. Specifically, we use the map $\Psi(x_1, \dots, x_d) \defeq (x_1, \dots, x_d, x_1^2 + \dots + x_d^2)$ to lift $(d-1)$-spheres to $d$-dimensional affine spaces and then apply Lund's result to the lifted point set. 

\subsubsection{Counting tuples with affine and linear degeneracies}

Let $A(n, d, k, r)$ and $L(n, d, k, r)$ denote the number of $r$-tuples of points in $[n]^d$ that lie in $k$-dimensional affine and linear subspaces, respectively. Note that $L(n,d,k,r)$ counts the number of $d\times r$ matrices with entries in $[n]$ of rank at most $k$. Katznelson~\cite{Kat} determined the asymptotic count of integral matrices of fixed rank as $n$ goes to infinity. As an almost immediate consequence, we obtain the correct order of growth for both $A(n,d,k,r)$ and $L(n, d, k, r)$.

\begin{prop}
\label{prop:affine-estimate}
    Let $d, k, r \in \N$ with $k < d$ and $r \ge k + 2$. Then $A(n, d, k, r) = \Theta(a_{d, k, r}(n))$, where
    \begin{equation}
    \label{eqn:adef}
        a_{d, k, r}(n) = \begin{cases}
            n^{d+k\cdot\max\{(r-1),d\}}, &\text{ if }r-1\neq d, \\
            n^{d(k+1)} \log n, &\text{ otherwise}.
        \end{cases}
    \end{equation}
\end{prop}

The special cases $k = 1$ and $k = d-1$ of Proposition~\ref{prop:affine-estimate} have previously appeared in the literature in differing levels of generality. See for instance \cite[Lemma~2]{BW} and \cite{GK} for $k = 1$, and \cite[Proposition~4]{BW} and \cite[Lemma~3.3]{SW} for $k = d-1$. 

\begin{prop}
\label{prop:linear-estimate}
    Let $d, k, r \in \N$ with $k < \min\{d,r\}$. Then $L(n, d, k, r) = \Theta(l_{d, k, r}(n))$, where
    \begin{equation}
    \label{defn for l}
        l_{d, k, r}(n) = \begin{cases}
            n^{k\cdot\max\{r,d\}}, &\text{ if } r \neq d,\\
            n^{kd} \log n, &\text{ otherwise}.
        \end{cases}
    \end{equation}
\end{prop}

\subsection{Organisation}
The rest of the paper is organised as follows. We prove Theorems \ref{thm:affine} and \ref{thm:linear} in Section \ref{sec:affine}, Theorem \ref{thm:sphere-count} in Section \ref{sec:sphere}, and Theorem \ref{thm:circle-count} in Section \ref{sec:circle}. Finally, we discuss some open problems and directions for future research in Section \ref{sec:conclusion}.

\begin{note}
    It recently came to our knowledge that in an independent work, Dong and Xu~\cite{DX} proved an $n - o(n)$ lower bound on $f_{\text{sph}}(n, d)$ for all $d \ge 2$ using an algebraic construction. Corollary~\ref{thm:circle} gives a lower bound of the same order for $d = 2$ but with a worse constant. However, our construction is of independent interest since it is probabilistic -- see the discussion on random constructions in Section~\ref{sec:random}. For $d \ge 3$, Corollary~\ref{thm:sphere} gives much weaker lower bounds than Dong and Xu's $\Omega(n)$ bound. Conjecturally, however, our random construction attains the $\Omega(n)$ bound.
\end{note}

\section*{Acknowledgments}

The authors thank Rob Morris, Gabriel Dahia, Jo\~ao Pedro Marciano, and Dmitrii Zakharov for helpful conversations. We also thank Marcelo Campos for bringing the paper of Katznelson~\cite{Kat} to our attention. The research leading to this work began during the Randomness and Learning on Networks (RandNET) program, IMPA, Rio de Janeiro, 29 July to 30 August 2024. We are grateful to IMPA for their hospitality. The second author's visit was funded by the RandNET project, MSCA-RISE Programme, Grant agreement ID: 101007705. The first and third authors were supported by ERC Advanced Grant 883810. The first and second authors also thank the Clarendon Fund for funding their research.

\section{Avoiding affine and linear degeneracies}
\label{sec:affine}

Let $M(n, d, k, r)$ denote the number of integral matrices $A\in [-n,n]^{d \times r}$  of rank at most $k$. The following theorem concerning the order of growth of $M(n,d,k,r)$ as $n\to \infty$ is an immediate consequence of a result of Katznelson~\cite{Kat}. It follows from ~\cite[Theorem~1]{Kat} by noting that matrices with rank exactly $k$ constitute the main term and that $B(0, n) \subseteq [-n, n]^{d \times r} \subseteq B(0, n \sqrt{dr})$, where $B(\mathbf{x}, R)$ denotes the Euclidean ball of radius $R$ centered at $\mathbf{x} \in \R^{d \times r}$.

\begin{theorem}[Katznelson~\cite{Kat}]
\label{thm:Kat}
    Let $d,k,r\in \mathbb{N}$ with $k < \min \{d, r\}$. Then, $M(n, d, k, r)=\Theta(l_{d,k,r}(n))$, where $l_{d,k,r}$ is as defined in \eqref{defn for l}.
\end{theorem}

It is easy to see from the proof of \cite[Theorem~1]{Kat} that the above theorem also holds if we only allow the matrices to have entries in $[n]$ instead of $[-n,n] \cap \Z$. With this observation in mind, we are ready to prove our results on counting tuples in linear and affine spaces.

\begin{proof}[Proof of Proposition~\ref{prop:linear-estimate}]
    The set of $r$-tuples of points in $[n]^d$ lying on $k$-dimensional linear spaces is in bijective correspondence with the set of $d \times r$ matrices of rank at most $k$ with each entry in $[n]$. Then, the observation after Theorem~\ref{thm:Kat} implies $L(n, d, k, r) = \Theta(l_{d, k, r}(n))$. We remark that we do not need to use the observation for the upper bound; it follows from Theorem~\ref{thm:Kat} itself.
\end{proof}

\begin{proof}[Proof of Proposition~\ref{prop:affine-estimate}]
    We have $n^d$ choices for the first point in the tuple. Translating the space so that the chosen point becomes the new origin, the number of choices for the remaining $r-1$ points in the tuple is bounded above by the number of $(r-1)$-tuples of points in $([-n, n] \cap \Z)^d$ lying on $k$-dimensional linear spaces, which in turn equals $M(n, d, k, r-1)$. Then, recalling the definition of $a_{d,k,r}(\cdot)$ from \eqref{eqn:adef}, Theorem~\ref{thm:Kat} implies
    \begin{equation*}
        A(n, d, k, r) \le n^d \cdot M(n, d, k, r-1) = O(a_{d,k,r}(n)).
    \end{equation*}
    The lower bound follows using Proposition~\ref{prop:linear-estimate} by choosing the first point from the box $[\lfloor n/2 \rfloor ]^d$, translating the space so that the chosen point becomes the new origin, and considering the $(r-1)$-tuples of points in $[\lfloor n/2 \rfloor ]^d$ lying on $k$-dimensional linear spaces in the translated space.
\end{proof}

Theorems \ref{thm:affine} and \ref{thm:linear} follow from Propositions \ref{prop:affine-estimate} and \ref{prop:linear-estimate} immediately using Lemma \ref{lem:deletion}. Note that, to prove Theorem~\ref{thm:linear} when $r > d$, we use the counting result for $r = d$ as it yields a better lower bound on $f_{\textup{lin}}(n, d, k, r)$.

\section{Avoiding spherical degeneracies}
\label{sec:sphere}

Our primary goal in this section is to prove the upper bound on $S(n, d)$ in Theorem~\ref{thm:sphere-count}, and consequently deduce Corollary \ref{thm:sphere} using it. We will also prove the claimed lower bound on $S(n,d)$ in Theorem~\ref{thm:sphere-count} at the end of the section.

\subsection{Number of lattice points on spheres}

In the following two lemmas, we prove an upper bound on the number of lattice points on ellipses and then use it to derive a bound on the number of lattice points on spheres. Although they are similar to Lemmas~3.2 and 3.3 in \cite{Shef}, we include their proofs for completeness.

\begin{lemma}
\label{lem:ellipse}
    There is a positive constant $c$ such that for any integer $d \ge 2$, every ellipse in $\R^d$ contains $O(n^{c/\log \log n})$ points from $[n]^d$.
\end{lemma}

\begin{proof}
    We will prove this by induction on $d$. First, we consider the case $d = 2$. Let $E$ be an ellipse in $\R^2$. We may assume that it contains at least $5$ points from $[n]^2$ because otherwise the result holds trivially. Suppose that $E$ is defined by the equation
    \begin{equation}
    \label{eqn:ellipse}
        x^2 + bxy + cy^2 + dx + ey + f = 0, \text{ where } b,c,d,e,f \in \R \text{ and } c \ne 0.
    \end{equation}
    Plugging the five points into the above equation, we get a system of five linear equations in five variables with a unique solution. Since each of these points lies in $[n]^2$, the variables $b, c, d, e, f$ have rational values with the numerators and denominators (in the reduced form) equal in absolute value to $O(n^{\alpha})$, for some constant $\alpha > 0$. Multiplying \eqref{eqn:ellipse} by the least common multiple of the denominators of $b, c, d, e, f$, we obtain
    \begin{equation*}
        Ax^2 + Bxy + Cy^2 + Dx + Ey + F = 0,
    \end{equation*}
    where $A, B, C, D, E, F \in \Z$ with their absolute values equal to $O(n^{5\alpha})$. Now the above equation can be rewritten as
    \begin{equation*}
        AX^2 + BXY + CY^2 = F',
    \end{equation*}
    for some $F' \in \Z$, where $X = \Delta x + BE - 2CD$, $Y = \Delta y + BD - 2AE$, and $\Delta = B^2 - 4AC$ is the discriminant. Finally, note that the number of integer points on the above ellipse is $O(\mathbf{d}(\Delta))$; see \cite[Equation 11.9]{Iwa} for a precise formula. Here, $\mathbf{d}(\Delta)$ denotes the number of divisors of $\Delta$, and is $O(n^{c/\log \log n})$ for some positive constant $c$, by the divisor bound~\cite[Theorem~317]{HW}.

    Now suppose $d \ge 3$. Let $E$ be an ellipse in $\R^d$. Further, let $A$ be the $2$-dimensional affine subspace spanned by $E$ and let $W$ be the $2$-dimensional linear subspace obtained by translating $A$ so that it contains the origin. Then there exists $i \in [d]$ such that the standard basis vector $\mathbf{e}_i$ is not contained in $W$. Consider the projection map $\pi_i$ that projects points in $\R^d$ onto the hyperplane $P$ defined by $x_i = 0$. Then the projection $\pi_i(E)$ of $E$ is an ellipse. Moreover, the number of points from $[n]^d$ that lie on $E$ is bounded above by the number of points on $\pi_i(E)$ that lie on $\pi_i([n]^d)$. Identifying the hyperplane $P$ with $\R^{d-1}$, the result follows from the induction assumption.
\end{proof}

\begin{lemma}
\label{lem:sphere-count}
    There is a positive constant $c$ such that for any integers $d \ge 2$ and $1 \le k \le d-1$, every $k$-sphere in $\R^d$ contains $O(n^{k-1+c/\log \log n})$ points from $[n]^d$.
\end{lemma}

\begin{proof}
    We will prove this by induction on $k$. For $k = 1$, the result follows directly from Lemma~\ref{lem:ellipse}. Now suppose $2 \le k \le d-1$ and let $S$ be a $k$-sphere in $\R^d$. Let $i \in [d]$ be such that the standard basis vector $\mathbf{e}_i$ is not orthogonal to the affine space $A$ spanned by $S$. Such a vector exists since $A$ has dimension at least $3$. Now consider the slices of $S$ obtained by intersecting with the hyperplanes $H_j$ defined by $x_i = j$ for $j \in [n]$. Each of these slices is either empty, a singleton, or a $(k-1)$-sphere, and therefore contains $O(n^{k-2+c/\log \log n})$ points by the induction assumption. Summing up over the $n$ slices, we conclude that $S$ contains $O(n^{k-1+c/\log \log n})$ points from $[n]^d$.
\end{proof}

\subsection{An incidence result}

Let $d \ge 2$ and $1 \le k \le d-1$. Further, let $\alpha\in(0,1)$. Suppose $P$ is a finite set of points in $\R^d$. A $k$-dimensional affine space $A$ is said to be $\alpha$-nondegenerate with respect to $P$ if $P \cap A$ is nonempty and at most $\alpha |P \cap A|$ points of $P$ lie on any $(k - 1)$-dimensional affine subspace of $A$. Recall that a $k$-flat refers to a $k$-dimensional affine subspace. We are now ready to state the following bound on point flat incidences due to Lund~\cite[Theorem~4]{bendlund}.

\begin{lemma}[Lund~\cite{bendlund}]
\label{lem:Lund}
    For each integer $k\geq 1$, real $\alpha\in(0,1)$ and integer $r>k$, the number of $\alpha$-nondegenerate, $r$-rich $k$-flats is $O(n^{k+1}r^{-k-2}+n^kr^{-k})$.
\end{lemma}

Similar to nondegenerate affine spaces, we say that a $k$-sphere $S$ is $\alpha$-nondegenerate with respect to $P$ if $P \cap S$ is nonempty and at most $\alpha |P \cap S|$ points of $P$ lie on any $(k - 1)$-sphere contained in $S$. Now suppose $P$ is a finite set of points in $\R^d$ such that $P$ is contained in some $k$-sphere but not in any $(k-1)$-sphere, where $1 \le k \le d-1$. Then such a $k$-sphere $S$ is unique and we say that $S$ is the spherical span of $P$, denoted by $\text{Span}_{\text{sph}}(P)$. Moreover, we say that the spherical dimension of $P$ is equal to $k$. We will need the following lemma, which asserts that for a nondegenerate $k$-sphere $S$, a constant proportion of $(k+3)$-tuples in $S$ have spherical dimension $k$.

\begin{lemma}
\label{lem:nondegen}
    Let $d \ge 2$ and $1 \le k \le d-1$. Suppose $S$ is a $k$-sphere in $\R^d$ such that $|S \cap [n]^d| = m$ and $S$ is $\frac{1}{2}$-nondegenerate with respect to $[n]^d$. Then the number of $(k+3)$-tuples in $S$ with spherical dimension $k$ is at least $m^{k+3}/2^{k+3}$.
\end{lemma}

\begin{proof}
    We use the following procedure to construct many $(k+3)$-tuples with spherical dimension $k$. Let $T = \emptyset$. In each step, we choose a point from $(S \cap [n]^d) \setminus \text{Span}_{\text{sph}}(T)$ and add it to $T$. We continue this process until $|T| = k+2$. Note that the tuple $T$ obtained after this procedure has spherical dimension $k$. There are $m$ choices for the first point that is added to $T$. Moreover, it follows from the nondegeneracy condition that there are at least $m/2$ choices for the point being added to $T$ in each subsequent step. Finally, we add one more point to this tuple from $(S \cap [n]^d) \setminus T$. The number of choices for this last point is $m - (k+2)$.
    
    Now, the $\frac{1}{2}$-nondegeneracy of $S$ implies that $m = |S \cap [n]^d| \ge 2(k+1) \ge 4$, which further implies
    \begin{equation*}
        m - (k+2) = m - (k+1) - 1 \ge m - \frac{m}{2} - \frac{m}{4} = \frac{m}{4}.
    \end{equation*}
    Multiplying the number of choices across all steps implies that the number of $(k+3)$-tuples in $S \cap [n]^d$ with spherical dimension $k$ is at least $m^{k+3}/2^{k+3}$.
\end{proof}

\subsection{An upper bound on $S(n,d)$}

We are now ready to prove the claimed upper bound on $S(n,d)$.

\begin{proof}[Proof of Theorem~\ref{thm:sphere-count} (Upper Bound)]
    Throughout the proof, we assume that $n$ is large enough and $c > 0$ is a constant for which Lemma~\ref{lem:sphere-count} holds. Let $P = [n]^d$. Further, let $\mathcal{S}$ be the set of $(d-1)$-spheres $S$ that contain at least $d+1$ points from $P$ and $S \cap P$ is not contained in any $(d-2)$-sphere. We partition $\mathcal{S}$ into $\mathcal{N}$ and $\mathcal{D}$, where $\mathcal{N}$ is the set of $\frac{1}{2}$-nondegenerate spheres in $\mathcal{S}$ and $\mathcal{D}$ is the set of $\frac{1}{2}$-degenerate spheres in $\mathcal{S}$. Let $\Psi: \R^d \rightarrow \R^{d+1}$ be the lifting map defined by
    \begin{equation*}
        \Psi(x_1, \dots, x_d) = (x_1, \dots, x_d, x_1^2 + \dots + x_d^2).
    \end{equation*}
    We apply the map $\Psi$ to the point set $P$ as well as to the spheres in $\mathcal{N}$ to obtain $\Psi(P)$ and $\Psi(\mathcal{N})$, respectively. Note that the map $\Psi$ lifts $(d-1)$-spheres to $d$-flats and also preserves nondegeneracy, namely, $d$-flats in $\Psi(\mathcal{N})$ are $\frac{1}{2}$-nondegenerate with respect to $\Psi(P)$.
    
    Let $T$ be the set of $(d+2)$-tuples in $[n]^d$ that lie on $(d-1)$-spheres. We will bound $|T|$ by partitioning $T$ into four parts based on whether the spherical span of the tuple $t$ is: (1) a $(d-2)$-sphere, (2) a nondegenerate $(d-1)$-sphere, (3) a degenerate $(d-1)$-sphere with an arbitrarily chosen rich sub-$(d-2)$-sphere also degenerate, or (4) a degenerate $(d-1)$-sphere with an arbitrarily chosen rich sub-$(d-2)$-sphere nondegenerate.
    
    We begin by bounding the first part. Let $T'$ be the set of tuples in $T$ that are not contained in any $(d-2)$-sphere. Then Lemma~\ref{lem:sphere-count} implies
    \begin{equation}
    \label{eqn:T-T'}
        |T \setminus T'| \lesssim \binom{d}{2} (n^d)^d (n^{d-3+c/\log \log n})^2 = O(n^{d(d+2) - 6 + 2c/\log \log n}),
    \end{equation}
    since each tuple in $T \setminus T'$ must be contained in a $(d-2)$-sphere.

    We now bound the second part, namely, we bound the number of tuples $t \in T'$ such that the unique $(d-1)$-sphere that $t$ determines is in $\mathcal{N}$. This is bounded above by the number of $(d+2)$-tuples in $\Psi(P)$ lying in a $d$-flat $H \in \Psi(\mathcal{N})$, which is in turn bounded above by
    \begin{equation*}
        \sum_{H \in \Psi(\mathcal{N})} |H \cap \Psi(P)|^{d+2}.
    \end{equation*}
    For each integer $j \ge 0$, let $\mathcal{H}_j$ denote the set of $d$-flats $H$ in $\Psi(\mathcal{N})$ such that $2^j \le |H \cap \Psi(P)| < 2^{j+1}$. Lemma~\ref{lem:sphere-count} implies that each $H \in \Psi(\mathcal{N})$ satisfies $|H \cap \Psi(P)| \le Cn^{d-2+c/\log \log n}$ for some constant $C > 0$. Therefore, we have
    \begin{equation}
    \label{eqn:richness-sum}
        |\{t \in T': \text{Span}_{\text{sph}}(t) \in \mathcal{N}\}| \le \sum_{H \in \Psi(\mathcal{N})} |H \cap \Psi(P)|^{d+2} \le \sum_{j= \lfloor \log_2 (d+2) \rfloor}^{\lfloor \log_2 (Cn^{d-2+c/\log \log n}) \rfloor} |\mathcal{H}_j| \cdot 2^{(j+1)(d+2)}.
    \end{equation}
    Now we use Lemma \ref{lem:Lund} to get that
    \begin{equation*}
        |\mathcal{H}_j| \lesssim |\Psi(P)|^{d+1} (2^j)^{-(d+2)} + |\Psi(P)|^d (2^j)^{-d} = n^{d(d+1)} 2^{-j(d+2)} + n^{d^2} 2^{-jd}.
    \end{equation*}
    Plugging the above estimate in \eqref{eqn:richness-sum}, we obtain
    \begin{align}
    \label{eqn:N}
        |\{t \in T': \text{Span}_{\text{sph}}(t) \in \mathcal{N}\}| &\lesssim \sum_{j = \lfloor \log_2 (d+2) \rfloor}^{\lfloor \log_2 (Cn^{d-2+c/\log \log n}) \rfloor} \brax{n^{d(d+1)} 2^{-j(d+2)} + n^{d^2} 2^{-jd}} \cdot 2^{(j+1)(d+2)} \nonumber\\
        &= 2^{d+2} \cdot \sum_{j = \lfloor \log_2 (d+2) \rfloor}^{\lfloor \log_2 (Cn^{d-2+c/\log \log n}) \rfloor} \brax{n^{d(d+1)} + n^{d^2} 2^{2j}} \nonumber\\
        &= O(n^{d(d+2)-4+2c/\log \log n} + n^{d(d+1)} \log n).
    \end{align}
    
    We now bound the number of tuples $t \in T'$ such that the unique $(d-1)$-sphere that $t$ determines is in $\mathcal{D}$. For each sphere $S \in \mathcal{D}$, let $B(S) \subset S$ be a $(d-2)$-sphere such that $|B(S) \cap [n]^d| > |S \cap [n]^d|/2$. Let $\mathcal{D}'$ be the set of spheres $S \in \mathcal{D}$ such that $B(S)$ is $\frac{1}{2}$-degenerate. Now if $d \ge 4$, then Lemma~\ref{lem:sphere-count} implies that any sphere in $\mathcal{D}'$ contains $O(n^{d-4 + c/\log \log n})$ points from $[n]^d$. And if $d = 3$, then any sphere in $\mathcal{D}'$ contains at most four points. Thus, we obtain the following bound on the third part.
    \begin{align}
    \label{eqn:D'}
        |\{t \in T': \text{Span}_{\text{sph}}(t) \in \mathcal{D}'\}| &\lesssim |\mathcal{D}'| \cdot (n^{d-4+c/\log \log n} + 1) \nonumber \\
        &\le n^{d(d+1)} (n^{d-4+c/\log \log n}+1) \nonumber \\
        &= O(n^{d(d+2)-4+c/\log \log n} + n^{d(d+1)}).
    \end{align}
    Finally, we bound the fourth part, namely, we bound the number of tuples $t \in T'$ such that the unique $(d-1)$-sphere that $t$ determines is in $\mathcal{D} \setminus \mathcal{D}'$. Let $S \in \mathcal{D} \setminus \mathcal{D}'$. Let $|B(S) \cap [n]^d| = a$ and $|(S \setminus B(S)) \cap [n]^d| = b$. Then $a > b$ by the definition of $B(S)$. Moreover, we have
    \begin{align}
    \label{eqn:ub}
        |\{t \in T': \text{Span}_{\text{sph}}(t) = S\}| &\le (d+2)! \binom{a+b}{d+2} - (d+2)! \binom{b}{d+2}\nonumber \\
        &\le \sum_{i = 1}^{d+2} \binom{d+2}{i} a^{d+2-i} b^i\nonumber \\
        &\le (d+2) \binom{d+2}{\lceil d/2 \rceil + 1} a^{d+1} b.
    \end{align}
    We define $T(S)$ as the set of $(d+2)$-tuples of points in $[n]^d$ such that the first $d+1$ points in the tuple span $B(S)$ and the last point is contained in $S \setminus B(S)$. Lemma~\ref{lem:nondegen} implies that the right-hand side of the \eqref{eqn:ub} is (up to a constant) bounded above by $|T(S)|$. For $S_1, S_2 \in \mathcal{D} \setminus \mathcal{D}'$, we have $T(S_1) \cap T(S_2) = \emptyset$ since $S$ can be recovered from any tuple in $T(S)$ by taking its span. Therefore, we have $\bigcup_{S \in \mathcal{D} \setminus \mathcal{D}'} T(S) \subseteq T''$, where $T''$ is the set of $(d+2)$-tuples of points in $[n]^d$ such that the first $d+1$ points in the tuple span a non-degenerate $(d-2)$-sphere. Thus, we obtain
    \begin{equation}
    \label{eqn:preD-D'}
        |\{t \in T': \text{Span}_{\text{sph}}(t) \in \mathcal{D} \setminus \mathcal{D}'\}| \le \sum_{S \in \mathcal{D} \setminus \mathcal{D}'} |T(S)| \le |T''| \le N \cdot n^d,
    \end{equation}
    where $N$ is the number of $(d+1)$-tuples of points in $[n]^d$ that span a non-degenerate $(d-2)$-sphere. By a similar argument as for bounding $|\{t \in T': \text{Span}_{\text{sph}}(t) \in \mathcal{N}\}|$, we have
    \begin{align*}
        N &\lesssim \sum_{j = \lfloor \log_2 (d+1) \rfloor}^{\lfloor \log_2 (C'n^{d-3+c/\log \log n}) \rfloor} \brax{n^{d^2} 2^{-j(d+1)} + n^{d(d-1)} 2^{-j(d-1)}} \cdot 2^{(j+1)(d+1)} \\
        &= 2^{d+1} \cdot \sum_{j = \lfloor \log_2 (d+1) \rfloor}^{\lfloor \log_2 (C'n^{d-3+c/\log \log n}) \rfloor} \brax{n^{d^2} + n^{d(d-1)} 2^{2j}}\\
        &= O(n^{d^2} \log n + n^{d(d+1)-6+2c/\log \log n}),
    \end{align*}
    for some constant $C' > 0$. Using the above estimate in \eqref{eqn:preD-D'}, we obtain
    \begin{equation*}
        |\{t \in T': \text{Span}_{\text{sph}}(t) \in \mathcal{D} \setminus \mathcal{D}'\}| = O(n^{d(d+2)-6+2c/\log \log n} + n^{d(d+1)} \log n).
    \end{equation*}
    Now adding \eqref{eqn:N} and \eqref{eqn:D'} to the above inequality yields
    \begin{equation*}
        |T'| = O(n^{d(d+2)-4+2c/\log \log n} + n^{d(d+1)} \log n).
    \end{equation*}
    Finally, adding \eqref{eqn:T-T'} to the above inequality yields the desired result.
\end{proof}

Note that the number of $(d+2)$-tuples of points that lie on a $d-1$-dimensional hyperplane in $[n]^d$ is $O(n^{d^2+d-1})$ by Proposition \ref{prop:affine-estimate}. Corollary \ref{thm:sphere} now follows immediately from Theorem \ref{thm:sphere-count} and the above fact using Lemma \ref{lem:deletion}.

\subsection{A lower bound on $S(n,d)$}

Here, we present the proof of the claimed lower bound on $S(n,d)$.

\begin{proof}[Proof of Theorem~\ref{thm:sphere-count} (Lower Bound)]
    Let $d \ge 3$. Fix a centre $x_0 \in [n]^d$ with each coordinate between $\lfloor n/4 \rfloor + 1$ and $\lceil 3n/4 \rceil$ (both ends inclusive). For each integer $m$ with $1 \le m \le \lfloor n/4 \rfloor^2$, let $S_m$ be the $(d-1)$-sphere centered at $x_0$ with radius $\sqrt{m}$. Let $\mathcal{S} := \{S_m: m \in [\lfloor n/4 \rfloor^2]\}$. We partition $\mathcal{S}$ into $\mathcal{N}$ and $\mathcal{D}$, where $\mathcal{N}$ is the set of $\frac{1}{2}$-nondegenerate spheres in $\mathcal{S}$. If $S \in \mathcal{D}$, then Lemma~\ref{lem:sphere-count} implies $|S \cap [n]^d| \le 2 \cdot C n^{d-3+c/\log \log n}$ for some constant $c$ and $C$. Further, note that the lattice points inside the punctured infinity ball of radius $\lfloor n/4 \rfloor/ \sqrt{d}$ centered at $x_0$ are contained inside $(\bigcup_{S \in \mathcal{S}} S) \cap [n]^d$. Therefore, $|(\bigcup_{S \in \mathcal{S}} S) \cap [n]^d| \ge (2 \lfloor n/4 \rfloor /\sqrt{d} - 2)^d - 1$, and consequently
    \begin{align}
    \label{eqn:lower}
        \sum_{S \in \mathcal{N}} |S \cap [n]^d| &\ge (2 \lfloor n/4 \rfloor /\sqrt{d} - 2)^d - 1 - \sum_{S \in \mathcal{D}} |S \cap [n]^d| \nonumber \\
        &\ge (2 \lfloor n/4 \rfloor /\sqrt{d} - 2)^d - 2C\lfloor n/4 \rfloor^2 n^{d-3+c/\log \log n} - 1.
    \end{align}
    Suppose that $n$ is sufficiently large so that the right-hand side of \eqref{eqn:lower} is bounded below by $C' n^d$ for some constant $C'$. Now Lemma~\ref{lem:nondegen} implies that the number of $(d+2)$-tuples of points in $[n]^d$ that span a sphere in $\mathcal{S}$ is bounded below by
    \begin{equation*}
        \sum_{S \in \mathcal{N}} \frac{|S \cap [n]^d|^{d+2}}{2^{d+1}} \ge \frac{|\mathcal{N}|}{2^{d+1}} \brax{\frac{\sum_{S \in \mathcal{N}} |S \cap [n]^d|}{|\mathcal{N}|}}^{d+2} \ge \frac{(C'n^d)^{d+2}}{(2n^2)^{d+1}} = \Omega(n^{d^2-2}).
    \end{equation*}
    Finally, note that there are $\lceil n/2 \rceil^d = \Omega(n^d)$ choices for the centre $x_0$. Therefore, we conclude $S(n, d) = \Omega(n^{d^2+d-2})$.
\end{proof}

\section{No-four-on-a-circle}
\label{sec:circle}

In this section, we will estimate the number of concyclic and collinear tuples of points in $[n]^2$.

\subsection{Counting collinear tuples}

Guy and Kelly~\cite{GK} determined the asymptotic count of collinear triples of points in $[n]^2$. For $r > 3$, we use a similar method to determine the asymptotic count of collinear $r$-tuples of points in $[n]^2$. Recall that, by a tuple, we mean an ordered set.

\begin{prop}
\label{prop:line-count}
    Let $r > 3$ be an integer. Then
    \begin{equation*}
        A(n, 2, 1, r) = \frac{2(r+3)}{r+1}\cdot\frac{\zeta(r-2)}{\zeta(r-1)} \cdot n^{r+1} + O(n^r \log n),
    \end{equation*}
    where $\zeta( \cdot )$ denotes the Riemann zeta function.
\end{prop}

\begin{proof}
    Let $m = n-1$. The number of collinear $r$-tuples of points in $[n]^2$ equals the number of collinear $r$-tuples of points in $[m]_0^2$. The number of collinear $r$-tuples on vertical and horizontal lines is equal to
    \begin{equation*}
        2n \cdot r! \cdot \binom{n}{r} = 2n^{r+1} + O(n^r).
    \end{equation*}
    Further, the number of collinear $r$-tuples on lines parallel to $x \pm y = 0$ is equal to
    \begin{equation*}
        2 \brax{2 \cdot \sum_{i = r}^{n-1} r! \cdot \binom{i}{r} + r! \cdot \binom{n}{r}} = \frac{4}{r+1} \cdot n^{r+1} + O(n^r).
    \end{equation*}
    Note that any other line passing through at least two points in $[m]_0^2$ has the form $bx \pm ay = c$, where $(a, b) = 1$, $c \in \Z$, and $a, b$ are distinct positive integers strictly smaller than $n$. We restrict our attention to $r$-tuples of points on lines of the form $bx - ay = c$ with $b < a$. Fix $a, b \in \N$ with $(a, b) = 1$ and $0 < b < a < n$. For $c \in \Z$, let $\ell_c$ denote the line $bx - ay = c$. Then the set of lines that intersect $[0,m]^2$ is $L := \{\ell_c: -am \le c \le bm\}$. Let $c \in \Z$ be such that $\ell_c \in L$. Then the number of $r$-tuples of points in $\ell_c$ is equal to
    \begin{equation}
    \label{eqn:rtups}
        |\ell_c \cap [m]_0^2|^r + O(|\ell_c \cap [m]_0^2|^{r-1}).
    \end{equation}
    The length of the projection of the line segment $\ell_c\cap [0,m]^2$ onto the $x$-axis differs from $a\cdot|\ell_c \cap [m]_0^2|$ by at most $O(1)$, and so we have that
    \begin{equation*}
        |\ell_c \cap [m]_0^2| = \begin{cases}
            \frac{c+am}{ab} + O(1), \quad & \text{if } -am \le c < m(b-a),\\
            \frac{m}{a} + O(1), \quad &\text{if } m(b-a) \le c \le 0,\\
            \frac{bm-c}{ab} + O(1), \quad &\text{otherwise}.
        \end{cases} 
    \end{equation*}
    Plugging the above estimate in \eqref{eqn:rtups} and summing over all possible values of $c$, we conclude that the total number of collinear $r$-tuples of points lying on some line in $L$ is equal to
    \begin{align*}
        2 \cdot &\sum_{c = -am}^{m(b-a)-1} \left[\brax{\frac{c+am}{ab}}^r + O\brax{\brax{\frac{c+am}{ab}}^{r-1}}\right] + \sum_{c=m(b-a)}^{0} \left[\brax{\frac{m}{a}}^r + O\brax{\brax{\frac{m}{a}}^{r-1}}\right]\\
        &= \frac{2}{r+1} \cdot \frac{b}{a^r} \cdot m^{r+1} + \frac{b}{a^{r-1}} \cdot O(m^r) + \frac{a-b}{a^r} \cdot m^{r+1} + \frac{a-b}{a^{r-1}} \cdot O(m^r)\\
        &= \frac{2}{r+1} \cdot \frac{b}{a^r} \cdot m^{r+1} + \frac{a-b}{a^r} \cdot m^{r+1} + \frac{1}{a^{r-2}} \cdot O(m^r).
    \end{align*}
    Now summing over all possible values of $a$ and $b$, we conclude that the number of collinear $r$-tuples of points in $[m]_0^2$ lying on some line with slope between $0$ and $1$, is equal to
    \begin{align*}
        \sum_{a = 2}^{m} \sum_{1 \le b < a:\; (a,b) = 1} &\left[\frac{2}{r+1} \cdot \frac{b}{a^r} \cdot m^{r+1} + \frac{a-b}{a^r} \cdot m^{r+1} + \frac{1}{a^{r-2}} \cdot O(m^r)\right]\\
        &= \sum_{a=2}^{m} \left[\brax{\frac{r+3}{r+1}} \frac{\phi(a)}{2a^{r-1}} \cdot m^{r+1} + \frac{\phi(a)}{a^{r-2}} \cdot O(m^r)\right]\\
        &= \frac{r+3}{2(r+1)} \brax{ \sum_{a=1}^{\infty} \frac{\phi(a)}{a^{r-1}} - 1} \cdot m^{r+1} + O(m^r \log m)\\
        &= \frac{r+3}{2(r+1)} \brax{ \frac{\zeta(r-2)}{\zeta(r-1)} - 1} \cdot n^{r+1} + O(n^r \log n).
    \end{align*}
    Here $\phi(\cdot)$ denotes Euler's totient function. We multiply the above count by a factor of $4$ to account for collinear $r$-tuples on lines with slope greater than $1$ as well as lines with negative slope. Finally, we add the count of collinear $r$-tuples on vertical and horizontal lines, as well as lines parallel to $x \pm y = 0$, to obtain the desired result.
\end{proof}

\subsection{Counting cyclic quadrilaterals}

We will now estimate the number of cyclic quadrilaterals with vertices in $[n]^2$. It turns out that almost all cyclic quadrilaterals with vertices in the grid $\mathbb{Z}^2$ are symmetric, that is, they are isosceles trapezia. Huxley and Konyagin~\cite{HK} have demonstrated this in the following setting. Consider equivalence classes of cyclic quadrilaterals with vertices in $\mathbb{Z}^2$ and circumradius at most $R$, under translation. They show that the number of such equivalence classes of asymmetric cyclic quadrilaterals is $O(R^{2+\frac{18}{29}+\varepsilon})$. In contrast, there are $cR^3 + O(R^2\log(R))$ equivalence classes of isosceles trapezia, where $c > 0$ is an absolute constant.

A closer inspection of Huxley and Konyagin's proof reveals that it implies a stronger result. Namely, there are $O(R^{2+\frac{18}{29}+\varepsilon})$ equivalence classes of asymmetric cyclic quadrilaterals with vertices in $\Z^2$ and diameter at most $O(R)$. This holds because their proof only uses the $O(R)$ bound on the distances between the vertices of the asymmetric cyclic quadrilaterals being counted. Accounting for the $O(R^2)$ possible translations, we obtain the following lemma. 

\begin{lemma}
\label{lem:asymmetric}
    For any set $A \subset \mathbb{Z}^2$ of diameter $R$ and $\varepsilon > 0$, there are $O(R^{4+\frac{18}{29}+\varepsilon})$ asymmetric cyclic quadrilaterals with vertices in $A$.
\end{lemma}

The following lemma (see \cite[Section~3.1]{GW}) will be useful for counting isosceles trapezia. We omit its proof since it is standard, although it follows quickly from the observation that there are $O(L+1)$ lattice points within distance $O(1)$ of any line segment of length $L$.

\begin{lemma}
\label{lem:polygon-lattice}
    Let $P$ be a convex polygon in $\R^2$. Then the number of lattice points in $P$ is $\mathrm{Area}(P) + O(\mathrm{Perimeter}(P)) + O(1)$. Moreover, the number of lattice points on $\partial P$ is $O(\mathrm{Perimeter}(P))$.
\end{lemma}

We are now ready to count isosceles trapezia with vertices in $[n]^2$.

\begin{lemma}
\label{lem:symmetric}
    The number of isosceles trapezia with vertices in $[n]^2$ is $\gamma n^5 + O(n^4 \log n)$, where
    \begin{align}
    \label{eqn:c}
        \gamma &= \frac{4}{15} + \sum_{a = 2}^{\infty} \sum_{1 \le b < a:\; (a,b) = 1} 2(3+(-1)^{a+b}) f(a,b),\\
        f(a, b) &= \frac{20 a^6 + 25 a^5 b - 7 a^4 b^2 + 28 a^3 b^3 - 20 a^2 b^4 + 3 a b^5 - b^6}{240 a^5 (a+b)^2 (a^2 + b^2)}. \nonumber
    \end{align}
\end{lemma}

\begin{proof}
    Let $m = n-1$. The number of isosceles trapezia with vertices in $[n]^2$ is equal to the number of isosceles trapezia with vertices in $[2m]_0^2$ such that each vertex has both coordinates even. We shall count the latter. Let $ABCD$ be an isosceles trapezoid with $A, B, C, D \in 2\cdot[m]_0^2$ and $AB$ parallel to $CD$. Then the midpoints of $AB$ and $CD$ both lie in $[2m]_0^2$. Therefore, the axis of symmetry of the isosceles trapezoid contains at least two points from $[2m]_0^2$. Let $\ell$ be a line containing at least two points from $[2m]_0^2$. Let $P_\ell$ be the convex body obtained as the intersection of $[0,2m]^2$ with its reflection under $\ell$. Further, let $m_\ell$ be the number of points from $2 \cdot [m]_0^2$ on one side of $\ell$ (not contained in $\ell$) in $P_\ell$ such that the projection of the point onto $\ell$ is a lattice point itself. Then the number of isosceles trapezia with $\ell$ as its axis of symmetry equals $\binom{m_\ell}{2}$.

    Therefore, the number of isosceles trapezia with a vertical or horizontal axis of symmetry equals
    \begin{equation*}
        2 \cdot 2 \cdot \sum_{i = 3}^{n} \binom{\lfloor \frac{i-1}{2} \rfloor n}{2} = \frac{n^5}{6} + O(n^4).
    \end{equation*}
    Similarly, the number of isosceles trapezia with axis of symmetry parallel to $x \pm y = 0$ equals
    \begin{equation*}
        2 \brax{2 \cdot \sum_{i = 3}^{n-1} \binom{\frac{i(i-1)}{2}}{2} + \binom{\frac{n(n-1)}{2}}{2}} = \frac{n^5}{10} + O(n^4).
    \end{equation*}
    Now fix integers $a, b$ with $0 < b < a \le 2m$ and $(a,b) = 1$. Similar to the proof of the lower bound in Theorem~\ref{thm:sphere-count}, we shall count the number of isosceles trapezia with their axis of symmetry parallel to $bx - ay = 0$. For $c \in \Z$, let $\ell_c$ denote the line $bx - ay = c$. The set of lines parallel to $bx - ay = 0$ that intersect $[2m]_0^2$ is $L := \{\ell_c: -2am \le c \le 2bm\}$. Due to the symmetry of the square, the number of isosceles trapezia with $\ell_c$ as its axis of symmetry is equal to the number of isosceles trapezia with $\ell_{2(b-a)m - c}$ as its axis of symmetry. We now divide our analysis into two cases depending on the parity of $a$ and $b$ modulo $2$.

    \noindent \textbf{Case 1}. Suppose exactly one of $a$ and $b$ is odd. Let $c \in \Z$. Then the set of points in $(2\Z)^2$ whose projection onto $\ell_c$ is a lattice point forms a sublattice $\mathcal{L}$ with determinant $4(a^2 + b^2)$ (see Figure~\ref{fig:abc(a)}). Let $P'_{\ell_c}$ be the half of $P_{\ell_c}$ that lies below $\ell_c$.
    
    We now compute the area of $P_{\ell_c}'$, which we shall use later to estimate $m_{\ell_c}$. Let $\theta \coloneqq \tan^{-1} \brax{\frac{b}{a}}$ and $h \coloneqq -\frac{c}{a}$ denote the inclination and $y$-intercept of the line $\ell_c$, respectively. The shape of $P_{\ell_c}'$ varies with the value of $c$. There are three different regimes, which we consider below. See Figure~\ref{fig:creg} for the shapes corresponding to these regimes and for aid with the computations below.

    \begin{figure}[!t]
        \begin{subfigure}{0.3\textwidth}
            \centering
            \includegraphics[width=\textwidth]{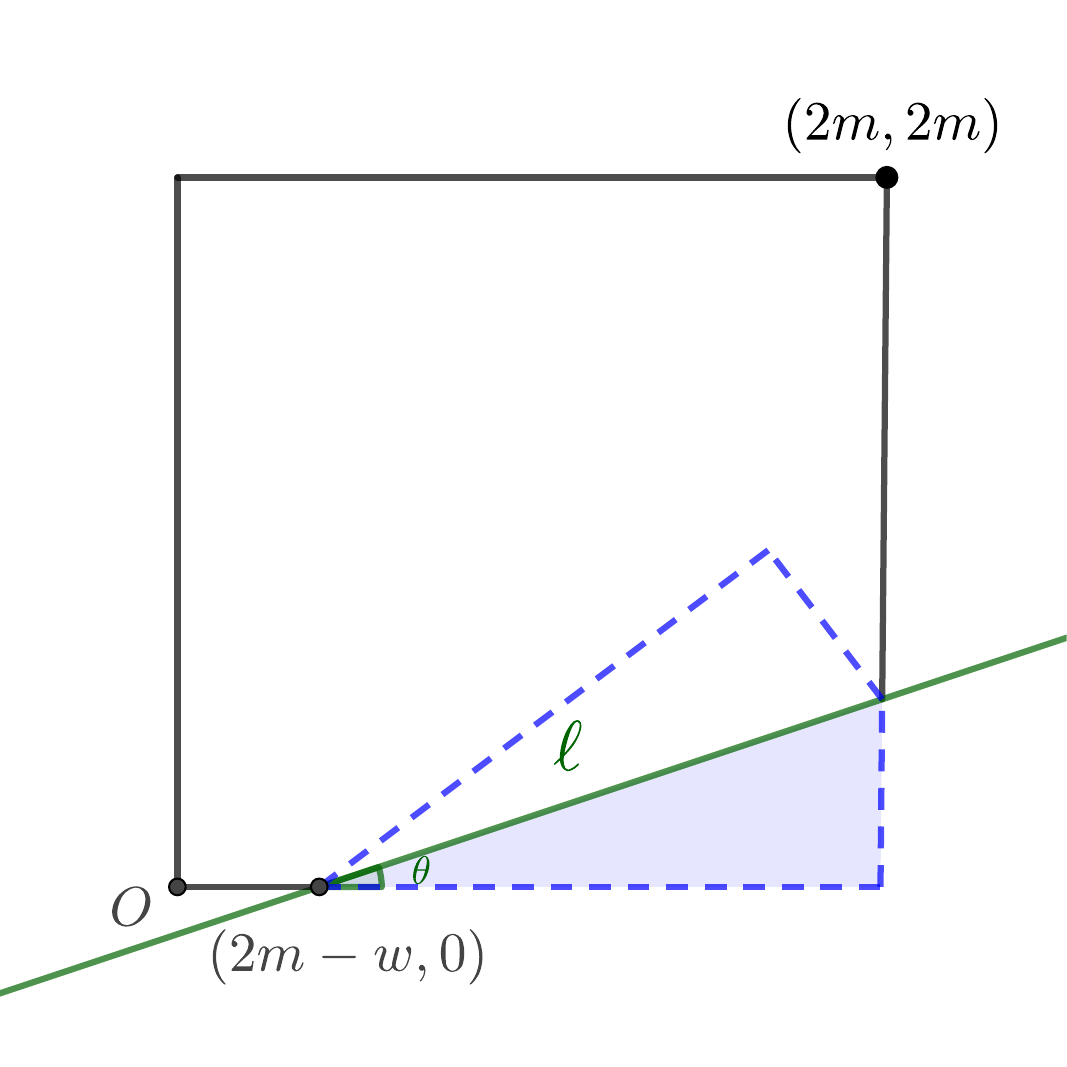}
        \end{subfigure}
        \begin{subfigure}{0.3\textwidth}
            \centering
            \includegraphics[width=\textwidth]{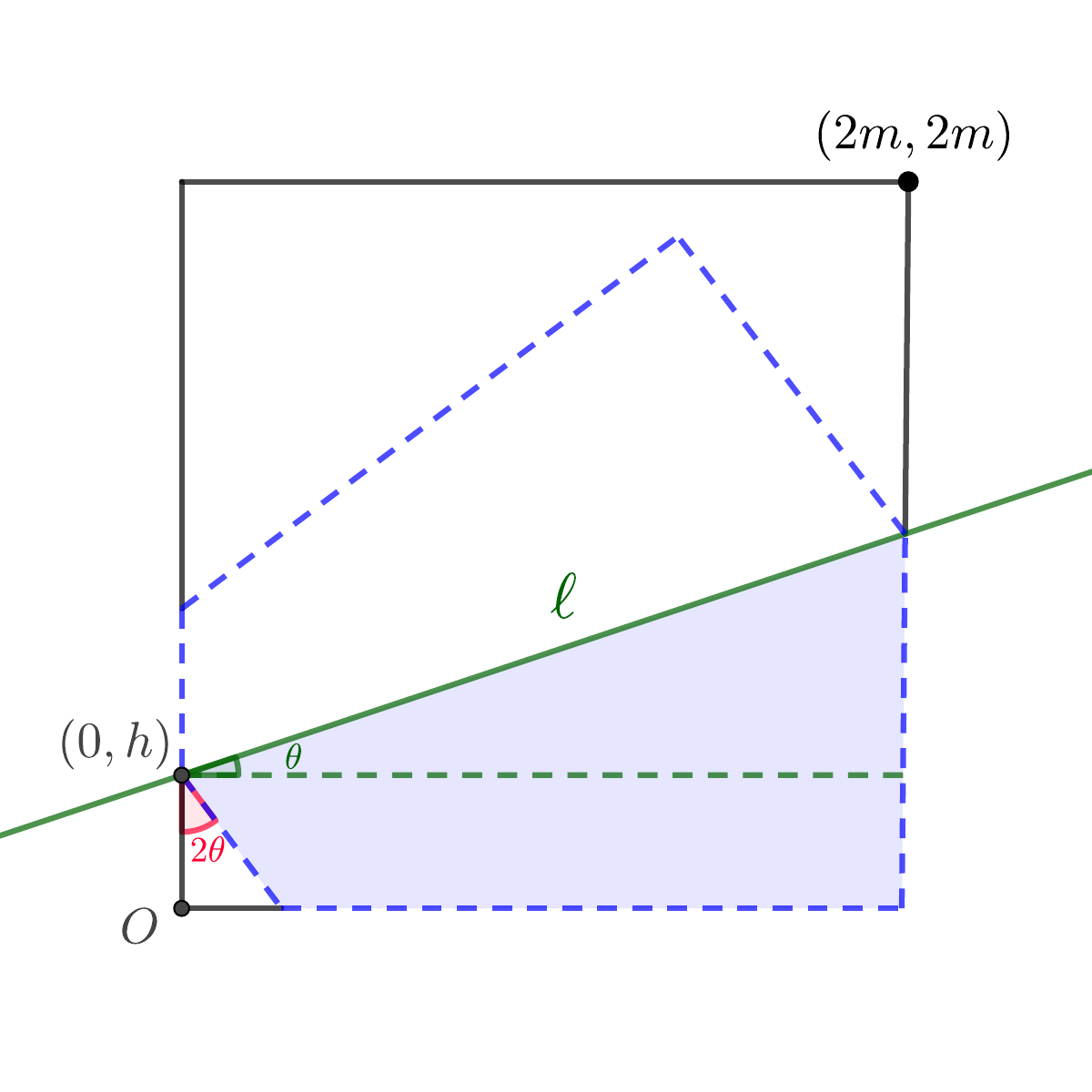}
        \end{subfigure}
        \begin{subfigure}{0.3\textwidth}
            \centering
            \includegraphics[width=\textwidth]{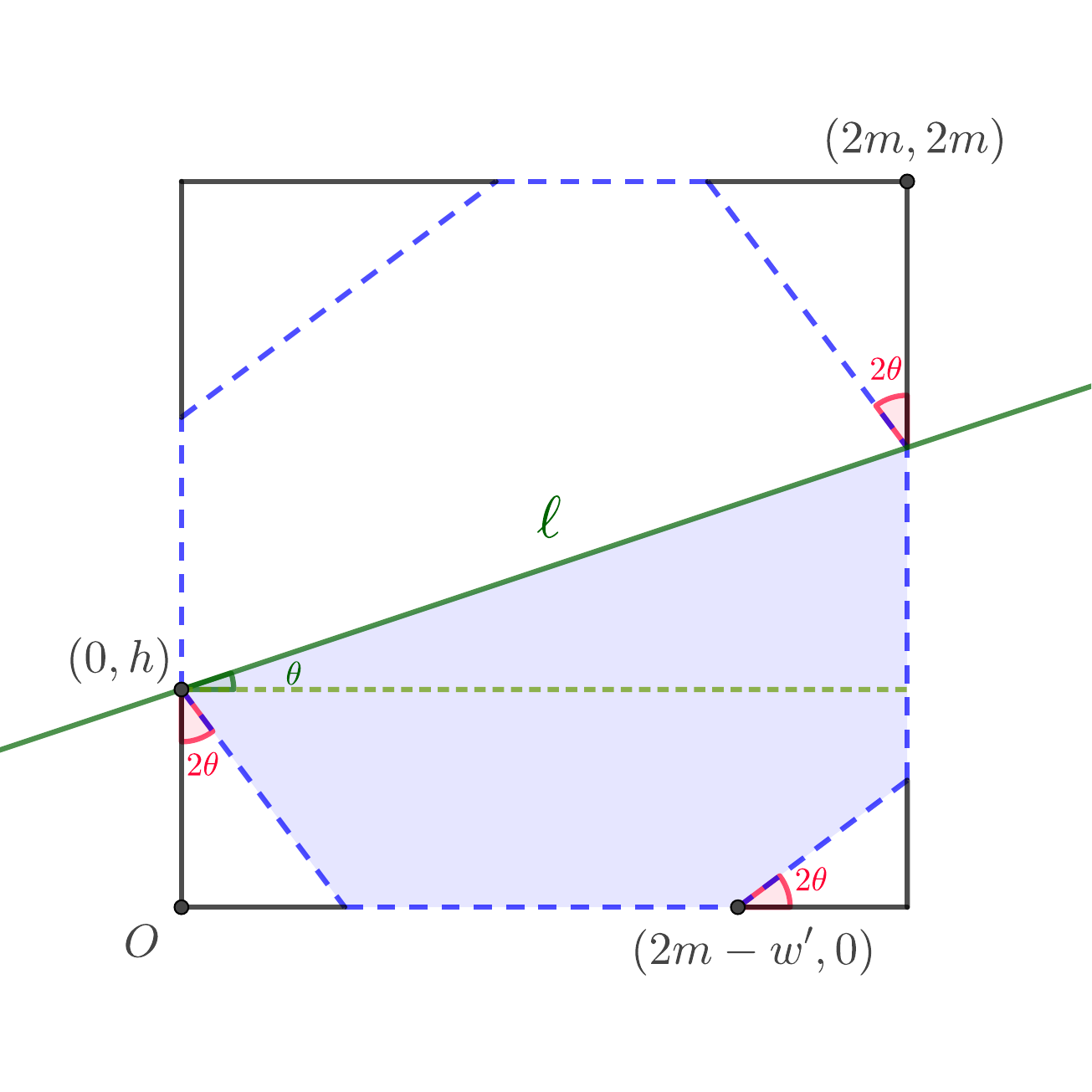}
        \end{subfigure}
        \caption{The shaded regions show $P'_{\ell_c}$ across the three different regimes for $c$.}
    \label{fig:creg}
    \end{figure}

    \noindent \textit{Case 1.1}. If $0 \le c \le 2bm$, $P'_{\ell_c}$ is a triangle with base $w = 2m - \frac{c}{b}$ and height $\frac{b}{a}w$.
    
    \noindent \textit{Case 1.2}. If $- \lfloor \frac{(a-b)^2}{a} m \rfloor \le c < 0$,  $P'_{\ell_c}$ comprises of a triangle with base $2m$ and height $\frac{b}{a}(2m)$, plus a rectangle with sides $h = -\frac{c}{a}$ and $2m$, minus a right-angled triangle with legs having lengths equal to $h$ and $h \tan(2\theta) = h \frac{2ab}{a^2-b^2}$.

    \noindent \textit{Case 1.3}. If $(b-a)m \le c < - \lfloor \frac{(a-b)^2}{a} m \rfloor$,  $P'_{\ell_c}$ comprises of a triangle with base $2m$ and height $\frac{b}{a}(2m)$, plus a rectangle with sides $h = -\frac{c}{a}$ and $2m$, minus two similar right-angled triangles with one of their acute angles equal to $2 \theta$. The first of these has a base (the leg adjacent to angle $2\theta$) of length $h$ as in Case~1.1 above, and the second has a base of length $w'$, say.
    
    In the rightmost subfigure in Figure~\ref{fig:creg}, equating the lengths of the hypotenuse of the right-angled triangle on the upper right corner and its reflection yields
    \begin{equation*}
        h + 2m \tan \theta = w' \tan(2\theta) + \brax{2m (1 - \tan \theta) - h} \sec(2\theta).
    \end{equation*}
    Plugging in $\tan \theta = \frac{b}{a}$ and simplifying, we obtain $w' = \frac{a}{b}h-\frac{(a-b)^2}{ab}m$.
    
    Combining the computations in the three regimes, we have
    \begin{equation}
    \label{eqn:area}
        \text{Area}(P'_{\ell_c}) = \begin{cases}
            \frac{b}{2a} w^2, \; &\text{if } 0 \le c \le 2bm,\\
            \frac{2b}{a}m^2 + 2mh - \frac{ab}{a^2-b^2} \cdot h^2, \; &\text{if } - \left\lfloor \frac{(a-b)^2}{a} m \right\rfloor \le c < 0,\\
            \frac{2b}{a}m^2 + 2mh - \frac{ab}{a^2-b^2} \brax{h^2 + w'^2}, \; &\text{if } (b-a)m \le c < - \left\lfloor \frac{(a-b)^2}{a} m \right\rfloor,\\
        \end{cases}
    \end{equation}
    where $w = 2m - c/b$, $h = -c/a$, and $w'=\frac{a}{b}h-\frac{(a-b)^2}{ab}m$.  
    
    Further, we have $\text{Perimeter}(P'_{\ell_c}) = O(m)$. Note that one can choose the fundamental parallelogram of $\mathcal{L}$ to be a square with side length $2\sqrt{a^2 + b^2}$. Therefore, we may dilate $\R^2$ by a factor of $(2 \sqrt{a^2 + b^2})^{-1}$ and apply Lemma~\ref{lem:polygon-lattice} to the image of $P'_{\ell_c}$ to obtain
    \begin{equation*}
        m_{\ell_c} = \frac{\text{Area}(P'_{\ell_c})}{4(a^2+b^2)} + O\brax{\frac{\text{Perimeter}(P'_{\ell_c})}{\sqrt{a^2 + b^2}}} + O(1) = \frac{\text{Area}(P'_{\ell_c})}{4(a^2+b^2)} + O\brax{\frac{m}{a}}.
    \end{equation*}
    Summing over all possible values of $c$, we conclude that the total number of isosceles trapezia with vertices in $2 \cdot [m]_0^2$ and their axis of symmetry in $L$, is equal to
    \begin{align*}
        \sum_{c=-2am}^{2bm} \binom{m_{\ell_c}}{2} &= 2 \cdot \sum_{c=(b-a)m+1}^{2bm} \left[ \frac{\text{Area}(P'_{\ell_c})^2}{32(a^2+b^2)^2} + \frac{1}{a^3} \cdot O(m^3) \right] + \left[ \frac{\text{Area}(P'_{\ell_{(b-a)m}})^2}{32(a^2+b^2)^2} + \frac{1}{a^3} \cdot O(m^3) \right]\\
        &= \sum_{c=(b-a)m+1}^{2bm} \frac{\text{Area}(P'_{\ell_c})^2}{16(a^2+b^2)^2} + \frac{1}{a^2} \cdot O(m^4).
    \end{align*}
    Plugging in the formula for $\text{Area}(P'_{\ell_c})$ from \eqref{eqn:area} and summing using \texttt{Mathematica}, the above expression simplifies to
    \begin{equation*}
        f(a,b) \cdot m^5 + \frac{1}{a^2} \cdot O(m^4).
    \end{equation*}

    \begin{figure}[!t]
        \begin{subfigure}{0.31\textwidth}
            \centering
            \includegraphics[width=0.85\textwidth]{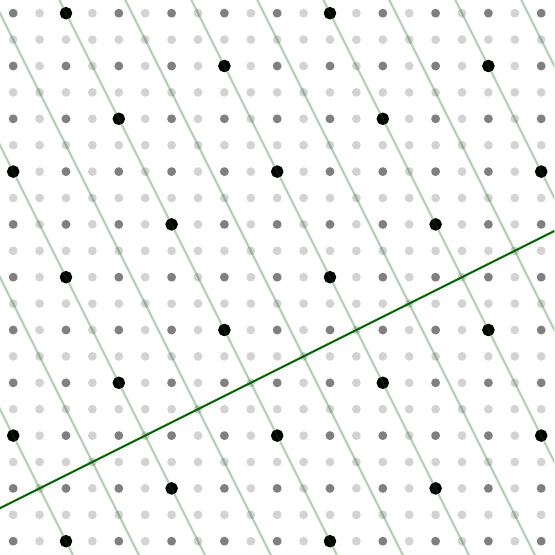}
            \caption{$a\not\equiv b \pmod 2$}
        \label{fig:abc(a)}
        \end{subfigure}
        \begin{subfigure}{0.31\textwidth}
            \centering
            \includegraphics[width=0.85\textwidth]{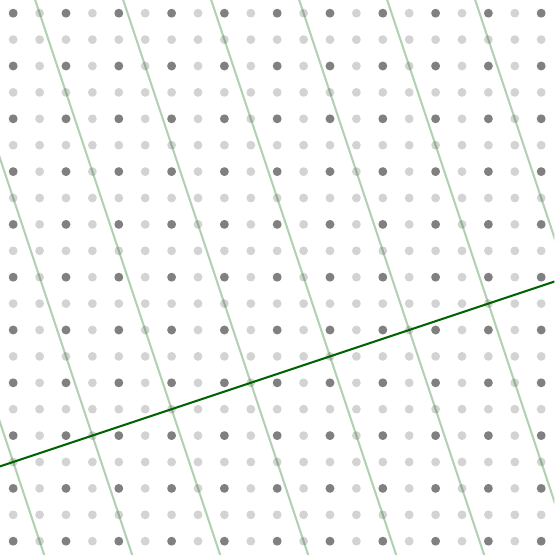}
            \caption{$a,b,c\equiv 1 \pmod2$}
        \label{fig:abc(b)}
        \end{subfigure}
        \begin{subfigure}{0.31\textwidth}
            \centering
            \includegraphics[width=0.85\textwidth]{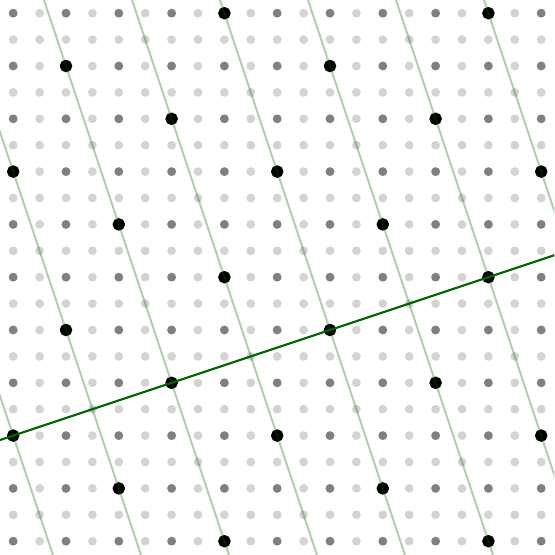}
            \caption{$a-b, c\equiv 0 \pmod2$}
        \label{fig:abc(c)}
        \end{subfigure}
        \caption{The sublattice $\mathcal{L} \subset (2\mathbb{Z})^2$ of points (marked in black) which project onto integer points on the line $\ell_c$ (shown in green) across different regimes for $a,b,c$.}
    \end{figure}

    \noindent \textbf{Case 2}. Suppose both $a$ and $b$ are odd. Let $c \in \Z$. If $c$ is odd, then there are no points in $(2\Z)^2$ whose projection onto $\ell_c$ is a lattice point (see Figure~\ref{fig:abc(b)}). Suppose $c$ is even. Then the set of points in $(2\Z)^2$ whose projection onto $\ell_c$ is a lattice point forms a sublattice $\mathcal{L}$ with determinant $2(a^2 + b^2)$ (see Figure~\ref{fig:abc(c)}). Using a similar argument as in case $1$, we conclude that the total number of isosceles trapezia with vertices in $2 \cdot [m]_0^2$ and their axis of symmetry in $L$ is equal to
    \begin{equation*}
        2 f(a,b) \cdot m^5 + \frac{1}{a^2} \cdot O(m^4).
    \end{equation*}
    
    Now summing over all possible values of $a$ and $b$, we conclude that the number of isosceles trapezia with vertices in $2 \cdot [m]_0^2$ and their axis of symmetry having slope between $0$ and $1$ equals
    \begin{align*}
        \sum_{a = 2}^{2m} \sum_{1 \le b < a:\; (a,b) = 1} &\frac{3 + (-1)^{a+b}}{2} \left[f(a,b) \cdot m^5 + \frac{1}{a^2} \cdot O(m^4)\right]\\
        &= \sum_{a = 2}^{2m} \sum_{1 \le b < a:\; (a,b) = 1} \frac{3 + (-1)^{a+b}}{2} \cdot f(a,b) \cdot m^5 + \sum_{a=2}^{2m} \frac{\phi(a)}{a^2} \cdot O(m^4)\\
        &= \sum_{a = 2}^{\infty} \sum_{1 \le b < a:\; (a,b) = 1} \frac{3 + (-1)^{a+b}}{2} \cdot f(a,b) \cdot n^5 + O(n^4 \log n).
    \end{align*}
    We multiply the above count by $4$ to account for isosceles trapezia with their axis of symmetry having a slope less than $0$ or greater than $1$. Finally, we add the count of isosceles trapezia with their axis of symmetry parallel to $x = 0$, $y = 0$, or $x \pm y = 0$. In the resulting count, we have also counted degenerate isosceles trapezia, namely, trapezia $ABCD$, for which $A$, $B$, $C$, and $D$ are collinear. Moreover, we have counted rectangles twice since they have two axes of symmetry. However, each of these counts is $\Theta(n^4 \log n)$. Subtracting these counts from the total count merges them with the error term, and we get the desired result.
\end{proof}

The sum appearing in the definition \eqref{eqn:c} of the constant $\gamma$ in Lemma~\ref{lem:symmetric} can be broken down into several parts, some of which we can express in terms of common number theoretic constants. However, since we are unable to express the entire sum in terms of such constants, we obtain rigorous numerical bounds on this constant in the lemma below.

\begin{lemma}
\label{lem:c}
    The constant $\gamma$ in Lemma~\ref{lem:symmetric} lies in the interval $(0.35974, 0.36017)$.
\end{lemma}

\begin{proof}
    For $N \ge 2$, let $s_N$ be the partial sum
    \begin{equation*}
        \sum_{a=2}^{N} \sum_{1 \le b < a:\; (a,b) = 1} 2(3+(-1)^{a+b}) f(a,b).
    \end{equation*}
    We compute $s_N$ in \texttt{Mathematica} for $N = 5500$ to an accuracy of $\varepsilon = 10^{-5}$ to obtain
    \begin{equation*}
        \abs{s_N - \widetilde{s}_N} \le \varepsilon,
    \end{equation*}
    where $\widetilde{s}_N =  0.09309$. Next, we bound the tail sum
    \begin{equation*}
        0 \le \gamma - s_N \le 8 \cdot \sum_{a=N+1}^{\infty} \frac{66\phi(a)}{240a^3} \le 2.2 \cdot \sum_{a=N+1}^{\infty} \frac{1}{a^2} \le \frac{2.2}{N}.
    \end{equation*}
    Combining the above estimates and using them in \eqref{eqn:c} yields
    \begin{equation*}
        \gamma - \frac{4}{15} \in \left[\widetilde{s}_N -\varepsilon, \widetilde{s}_N + \varepsilon+\frac{2.2}{N}\right].
    \end{equation*}
    Plugging in the values of $N$, $\widetilde{s}_N$, and $\varepsilon$, and simplifying yields the desired result.
\end{proof}

Applying Lemma~\ref{lem:asymmetric} to the point set $A = [n]^2$ implies that there are $O(n^{4+\frac{18}{29}+\varepsilon})$ asymmetric cyclic quadrilaterals with vertices in $[n]^2$. Combining this with Lemma~\ref{lem:symmetric} yields Theorem~\ref{thm:circle-count}.

\begin{proof}[Proof of Corollary~\ref{thm:circle}]
    By Proposition \ref{prop:line-count}, the number of unordered collinear quadruples is $\frac{7\pi^2}{360\zeta(3)}\cdot n^5+O(n^4\log(n))$. Combined with Theorem~\ref{thm:circle-count}, we get that the number of edges in the hypergraph encoding collinear and concyclic quadruples in $[n]^2$ is $(\gamma+\frac{7\pi^2}{360\zeta(3)}+o(1))n^5<cn^5$, with $c=0.51983$, using Lemma \ref{lem:c}. Finally, by Lemma \ref{lem:deletion}, we have that $f_{circ}(n)\geq\frac{3}{4}(\frac{1}{4c})^{\frac{1}{3}}n>\frac{7}{12}n$.
\end{proof}

\section{Concluding remarks}
\label{sec:conclusion}

To conclude, we highlight some open problems and potential directions for future research.

\subsection{Improving bounds on $S(n,d)$}
\label{sec:improvedspherecount}

As mentioned in the introduction, we conjecture that the following stronger upper bound on $S(n,d)$ holds.

\begin{conjecture}
\label{conj:sphere-counting}
    For any integer $d\ge 2$, we have $S(n,d)=O(n^{d^2+d-1})$.
\end{conjecture}

We have confirmed this conjecture for $d = 2$ in Theorem~\ref{thm:circle-count}. For $d \geq 3$, it is plausible that an even stronger bound holds, namely, $S(n, d) = n^{d^2+d-2+o(1)}$. If true, this would imply that the lower bound in Theorem~\ref{thm:sphere-count} is sharp up to a $o(1)$ term in the exponent. Exact computations of $S(n, d)$ for small values of $n$ and $d$ seem consistent with the above estimates.

Assuming joint spatial and modular equidistribution of lattice points on spheres of the form $x_1^2 + \dots + x_d^2 = m$, where $m \in \N$, yields a heuristic for the stronger upper bound mentioned above. Here, spatial equidistribution refers to equidistribution with respect to the Lebesgue measure on the sphere, and modular equidistribution refers to equidistribution of mod $k$ residues of the lattice points on the sphere over the mod $k$ solutions to the equation defining the sphere, where $k \in \N$. By joint equidistribution, we mean that modular equidistribution holds locally within small slices of the sphere. Spatial equidistribution of lattice points on spheres is known classically (see \cite[Section~11.6]{Iwa}), but not in conjunction with modular equidistribution. We remark that this approach fails in two dimensions due to a special symmetry corresponding to isosceles trapezoids.

Finally, we also highlight a connection to random matrix theory. Let $\mathbf{x}^{(1)}, \dots, \mathbf{x}^{(d+2)}$ be $d+2$ points in $[n]^d$. If they lie on a $(d-1)$-dimensional sphere, then the matrix
\begin{equation*}
    A(\mathbf{x}^{(1)}, \dots, \mathbf{x}^{(d+2)}) \coloneqq
    \begin{bmatrix}
        1 & x^{(1)}_1 & \dots & x^{(1)}_d &\|\mathbf{x}^{(1)}\|_2^2\\
        \vdots & \vdots & \ddots & \vdots & \vdots\\
        1 & x^{(d+2)}_1 & \dots & x^{(d+2)}_2 & \|\mathbf{x}^{(d+2)}\|_2^2
    \end{bmatrix}
\end{equation*}
is singular. For $i \in [d+2]$ and $j \in [d]$, let $X^{(i)}_j$ be iid random variables with distribution $\text{Uniform}([n])$. Then $S(n,d)$ is bounded above by
\begin{equation}
\label{eqn:prob}
    \mathbb{P}(\det(A(\mathbf{X}^{(1)}, \dots, \mathbf{X}^{(d+2)})) = 0) \cdot n^{d(d+2)}.
\end{equation}
For discrete random matrices with iid entries, the dominant contribution to the singularity probability is from the events that a row (or column) is zero or that two rows (or columns) are equal~\cite{JSS} (see also \cite{Tik}). Similarly, it might be the case that this is also the dominant term in the singularity probability in \eqref{eqn:prob}, which would imply an $O(n^{d^2+d})$ bound on $S(n, d)$. To prove Conjecture~\ref{conj:sphere-counting}, one would need better estimates on the error term.

\subsection{Extensible no-four-on-a-circle}

The classical no-three-in-line problem seeks to maximize the number of points that may be selected from $[n]^2$ while avoiding collinear triples of points. Erde asked whether one can select large subsets of $\Z^2$ that avoid collinear triples. Nagy, Nagy, and Woodroofe~\cite{NNW} showed that there exists such a set with $\Theta(n/\log^{1+\varepsilon} n)$ points in $[n]^2$ for all $n \in \N$. Considering a similar extensible version of the no-four-on-a-circle problem would be interesting.

\subsection{No-three-in-line in $[n]^d$}

Among the problems on avoiding affine degeneracies discussed previously, of special significance is the no-three-in-a-line problem in high dimension due to its connection with lower bounds on $r_3(n)$, the size of the largest $3$-AP-free subset of $[n]$ -- see Behrend's~\cite{Beh} and Elkin's construction~\cite{Elk}. As discussed in the introduction, we know that $f_{\text{aff}}(n,d,1,3) = \Omega(n^{d-2+\frac{2}{d+1}})$, and this is the best one can do using a strictly convex surface. Any improvement to this lower bound would correspond to an improvement over Elkin's lower bound on $r_3(n)$. Elkin's bound was improved recently by Elsholtz, Hunter, Proske, and Sauermann~\cite{EHPS}, but a good enough improvement to the lower bound $f_{\text{aff}}(n,d,1,3)$ might yield an even better bound on $r_3(n)$. In any case, we think that determining the order of growth $f_{\text{aff}}(n,d,1,3)$ is interesting even in its own right. Note that we only have the trivial upper bound $f_{\text{aff}}(n,d,1,3)=O(n^{d-1})$.

% References

\end{document}